\documentclass[11pt]{article}
\usepackage{amsmath,amssymb}
\usepackage[latin1]{inputenc}
\usepackage[dvips]{graphicx}

\topskip=\baselineskip  \headsep=0in
\newtheorem{theorem}{Theorem}[section]
\newtheorem{lemma}[theorem]{Lemma}

\newtheorem{corollary}[theorem]{Corollary}
\newtheorem{remark}[theorem]{Remark}

\newcommand{\qed}{\enspace\vrule height6pt width4pt depth2pt}

\newenvironment{proof}{\par\noindent{\bf Proof.}}{$\qed$\par\bigskip}

\newcommand{\vp}{\varphi}
\newcommand{\Char}{{\rm char}}

\newcommand{\GEN}[1]{\langle #1 \rangle}

\newcommand{\U}{{\cal U}}

\author{O. Broche \and E. Jespers \and C. Polcino Milies \and M. Ruiz}

\begin{document}

\title{ANTISYMMETRIC ELEMENTS IN GROUP RINGS II}

\date{}

%\address{Departamento de  M\'{e}todos Cuantitativos e
%Inform\'{a}ticos, Universidad Polit\'{e}cnica de Cartagena\\
% Cartagena, 30.203, Spain\\
%\email{manuel.ruiz@upct.es}}

\maketitle

\begin{abstract}
Let $R$ be a commutative ring, $G$ a group and $RG$
its group ring. Let $\vp : RG\rightarrow RG$ denote
the $R$-linear extension of an involution $\vp$
defined on $G$. An element $x$ in $RG$ is said to be
$\vp$-antisymmetric if $\vp (x) = -x$. A
characterization is given of when the
$\vp$-antisymmetric elements of $RG$ commute. This
is a completion of earlier work.
\end{abstract}

\textbf{keywords}: Involution; group ring;
antisymmetric elements.

\textbf{keywords}: 2000 Mathematics Subject
Classification: 16S34, 16W10,  20C07.

\section{Introduction.}
Throughout this paper   $R$ is a commutative ring
with identity, $G$ is a group and $\vp$ is an
involution on $G$.  Clearly $\vp$ can be extended
linearly to an involution $\vp :RG\rightarrow RG$ of
the group ring $RG$. Set $R_2=\{r\in R\mid 2r=0 \}$.
We denote by $(RG)_{\varphi}^{-}$ the Lie algebra
consisting of the $\varphi$-antisymmetric elements
of $RG$, that is $$(RG)^-_\varphi=\{\alpha\in RG|\;
\varphi(\alpha)=-\alpha\}.$$

For general algebras $A$ with an involution $\vp$,
we recall some important results that show that
crucial information of the algebraic structure of
$A$ can be determined by that of $(A)^-_\vp$ and the
latter has information that is determined by the
$\varphi$-unitary unit group $U_{\vp} (A)=\{ u\in
A\mid u \vp (u) =\vp (u)u=1\}$. By $U(A)$ we denote
the unit group of $A$. Amitsur in \cite{A} proves
that if $A_{\varphi}^{-}$ satisfies a polynomial
identity (in particular when $A_{\varphi}^{-}$ is
commutative) then $A$ satisfies a polynomial
identity. Gupta and Levin in \cite{gupta} proved
that for all $n\geq 1$ $\gamma_n(\U(A))\leq 1+
L_n(A)$. Here $\gamma_n(G)$ denotes the $n$th term
in the lower central series of the group $G$ and
$L_n(A)$ denotes the two sided ideal of $A$
generated by all Lie elements of the form
$[a_1,a_2,\dots,a_n]$ with $a_i\in A$ and
$[a_1]=a_1$, $[a_1,a_2]=a_1a_2-a_2a_1$ and
inductively
$[a_1,a_2,\dots,a_n]=[[a_1,a_2,\dots,a_{n-1}],a_n]$.
Smirnov and Zalesskii in \cite{ZS}, proved that, for
example, if the Lie ring generated by the elements
of the form $g+g^{-1}$ with $g\in \U(A)$ is Lie
nilpotent then $A$ is Lie nilpotent. In \cite{GM}
Giambruno and Polcino Milies show that if $A$ is a
finite dimensional semisimple algebra over an
algebraically closed field $F$ with $char(F)\neq 2$
then $\mathcal{U}_{\varphi}(A)$ satisfies a group
identity if and only if $(A)^{-}_{\varphi}$ is
commutative. Furthermore, if $F$ is a nonabsolute
field then $\mathcal{U}_{\varphi}(A)$ does not
contain a free group of rank $2$ if and only if
$(A)^{-}_{\varphi}$ is commutative. Giambruno and
Sehgal, in \cite{gia-seh}, showed that if $B$ is a
semiprime ring with involution $\varphi$, $B=2B$ and
$(B)^{-}_{\varphi}$ is Lie nilpotent then
$(B)^{-}_{\varphi}$ is commutative and $B$ satisfies
a polynomial identity of degree $4$.

Special attention has been given to the classical
involution $*$ on $RG$, that is, the $R$-linear map
defined by mapping $g\in G$ onto $g^{-1}$. In case
$R$ is a field of characteristic $0$ and $G$ is a
periodic group, Giambruno and Polcino Milies in
\cite{GM} described when $\mathcal{U}_{*}(RG)$
satisfies a group identity. Gon\c{c}alves and
Passman in \cite{GoPas} characterized when
$\mathcal{U}_{*}(RG)$ does not contain non abelian
free groups when $G$ is a finite group and $R$ is a
nonabsolute field. Giambruno and Sehgal, in
\cite{gs},  characterized when $(RG)_{*}^{-}$ is Lie
nilpotent provided  $R$ is a field of characteristic
$p\geq 0$, with $p\neq 2$.

Motivated by all these connections, in this paper we
deal with the question of when $(RG)^{-}_{\vp}$ is
commutative  for an arbitrary involution $\vp$ on
$G$. Let $G_{\varphi} =\{ g\in G \mid \varphi
(g)=g\}$ be the subset of $\varphi$-symmetric
elements of $G$, i.e. the set of elements of $G$
fixed by $\varphi$. The following complete answer is
obtained.

\begin{theorem}
Let $R$ be a commutative ring. Suppose $G$ is a
non-abelian group and $\vp$ is an involution on $G$.
Then, $(RG)^-_\vp$ is commutative if and only if one
of the following conditions holds:
\begin{enumerate}
\item
$K=\GEN{g\in G|\; g\not\in G_\vp}$ is abelian (and
thus  $G=K\cup Kx$, where $x\in G_\vp$, and
$\vp(k)=xkx^{-1}$ for all $k\in K$) and
$R_{2}^{2}=\{ 0 \}$.
\item
 $R_{2}=\{ 0\}$ and $G$ contains an abelian
subgroup of index $2$ that is contained in
$G_{\vp}$.
\item
$\Char(R)=4$, $|G'|=2$, $G/G'=(G/G')_\vp$, $g^2\in
G_\vp$ for all $g\in G$, and $G_{\vp}$ is
commutative in case  $R_{2}^{2}\neq \{ 0\}$.
\item $\Char (R)=3$, $|G'|=3$, $G/G'=(G/G')_\vp$ and
$g^3\in G_\vp$ for all $g\in G$.
\end{enumerate}
\end{theorem}

Clearly, as an $R$-module, $(RG)_{\varphi}^-$ is
generated by the set
\begin{eqnarray*}
{\mathcal S} &=&\{g-\vp(g)  \mid  \ g\in G\setminus
G_\vp\}\cup\{rg\mid \;g\in G_\vp, \;r\in R_2\}
\end{eqnarray*}
Therefore $(RG)^{-}_{\vp}$ is commutative if and
only if the elements in $\mathcal{S}$  commute.

This work is a continuation of the work started in
\cite{osn2} (for the classical involution),
\cite{varphi} and \cite{jes-osn-ruiz}. In the latter
one considers the involutions $\eta$ on $RG$
introduced by Novikov in \cite{novikov}: $ \eta
(\sum\limits_{g\in G}\alpha_g g)= \sum\limits_{g\in
G} \alpha_g\sigma(g)g^{-1}$, where
$\sigma:G\rightarrow \{\pm 1\}$ is a group
homomorphism. Unfortunately, in \cite{osn2,varphi}
the set ${\mathcal S}_{1} = \{ rg \mid r\in R_{2},\;
g\in G_{\varphi}\}$  was not included in the set
${\mathcal S}$. Therefore, the results given in
\cite{osn2,varphi} only deal  with commuting of
elements in the set ${\mathcal S}\setminus {\mathcal
S_{1}}$. Hence, provided $R_{2}=\{ 0\}$, there is a
complete characterization of when $(RG)_{\vp}^{-}$
is commutative in \cite{varphi} when $\Char (R)\neq
2,3$ and in \cite{osn2} when $\vp$ is the classical
involution and $\Char (R)\neq 2$. The case $\Char
(R)=3$ was left as an open problem in \cite{varphi},
and the case $\Char (R)=2$ has been dealt with in
\cite{osn1,JR} because then $(RG)_{\varphi}^{-}$
coincides with the set of $\varphi$-symmetric
elements of $RG$.

So, throughout the paper we assume $\Char (R)\neq
2$. The center of $G$ is denoted by $Z(G)$, the
additive commutator $\alpha \beta -\beta \alpha$ of
$\alpha , \beta \in RG$ is denoted $[\alpha , \beta
]$, and the multiplicative commutator
$ghg^{-1}h^{-1}$ of $g,h\in G$ is denoted by
$(g,h)$.

As mentioned above, the Theorem has been proved in
\cite{varphi} provided $R_{2}=\{ 0\}$ and $\Char
(R)\neq 3$. Theorem~\ref{resultchar4} shows the
result holds in case $R_{2}\neq 0$ and
Theorem~\ref{resultchar3} shows that it also holds
if $\Char (R)=3$.

\section{Rings with elements of additive order 2}

We begin with recalling some technical results from
\cite{varphi}. The first lemma shows that the group
generated by the non-fixed elements has index at
most $2$.

\begin{lemma}{\em\cite[Lemma 2.3]{varphi}}\label{K}
If $\vp$ is non-trivial then the subgroup
$K=\GEN{g\in G\mid g\not\in G_\varphi}$ has index at
most $2$ in $G$.
\end{lemma}

\begin{lemma}{\em\cite[Lemma 1.1 and Lemma 1.2]{varphi}}\label{lema1}
Let $R$ be a commutative ring with $\Char (R) \neq
2,3$. Let $g,h\in G\setminus G_{\vp}$ be two
non-commuting elements. If $(RG)^{-}_{\vp}$ is
commutative then one of the following conditions
holds
\begin{enumerate}
\item $gh\in G_{\vp}$, $hg\in G_{\vp}$, $h\vp (g)=g\vp (h)$ and
$\vp (h)g=\vp (g) h$.
%\item $gh\in G_{\vp}$, $hg\in
%G_{\vp}$, $h\vp (g)=\vp (g)h$.
%\item $gh\in
%G_{\vp}$, $hg=\vp (g)h=g\vp (h)$ and $\Char(R)=3$.
%\item $gh=h\vp(g)=\vp(g)\vp(h)$, $\vp(h)g=\vp(g)h$
%and $\Char(R)=3$.
%\item $gh=\vp(h)g=\vp(g)\vp(h)$,
%$h\vp(g)=g\vp(h)$ and $\Char(R)=3$.
%\item $hg\in G_{\vp}$, $gh=\vp (h)g=h\vp (g)$ and $\Char(R)=3$.
\item $gh=h\vp(g)=\vp(h)g=\vp(g)\vp(h)$ and $\Char(R)=4$.
\end{enumerate}
\end{lemma}

Note that,  if non-commuting elements $g,h\in
G\setminus G_{\vp}$ satisfy condition $(2)$ in the
lemma then $h^{-1}gh=\vp (g)$. Non-commutative
groups $G$  with an involution $\varphi$ such that
$h^{-1}gh\in \{ g,\; \varphi (g)\}$ for all $g,h\in
G$ have been described in
\cite[Theorem~III.3.3]{eric}. These are precisely
the groups $G$ with a unique non-trivial commutator
and that satisfy the {\it lack of commutativity
property} (``LC'' for short). The latter means that
for any pair of elements $g,h\in G$ it is the case
that $gh=hg$ if and only if either $g\in Z(G)$ or
$h\in Z(G)$ or $gh\in Z(G)$. It turns out
\cite[Proposition III.3.6]{eric} that such groups
are precisely those non-commutative groups  $G$ with
$G/Z(G)\cong C_{2}\times C_{2}$, where $C_{2}$
denotes the cyclic group of order $2$.

In the next lemma we give the  structure of the
group generated by two elements $g,h\in G\setminus
G_{\vp}$  satisfying $(2)$ of Lemma~\ref{lema1}.

\begin{lemma}{\em\cite[Lemma 3.1]{varphi}}\label{lema4}
Let $R$ be a commutative ring with $\Char(R)=4$.
Suppose $g,h\in G\setminus G_{\vp}$ are
non-commuting elements that  satisfy $(2)$ of
Lemma~\ref{lema1}. If $(RG)^{-}_\varphi$ is
commutative then the group
$H=\GEN{g,h,\varphi(g),\varphi(h)}=\GEN{g,h}$
satisfies the $LC$--property and has a unique
non-trivial commutator $s$  and the involution
restricted to $H$ is given by $\vp(h)=sh$ if $h\in
H\setminus Z(H)$ and $\vp (h)=h$ if $h\in Z(H)$.
\end{lemma}

%Because of the previous Lemma we need to know when
%$(RG)_{\vp}^{-}$ is commutative for a group $G$ with
%the $LC$--property and unique non-trivial commutator
%and an involution as described. The following
%proposition shows this is always the case (note that
%in this case $G_{\varphi}\subseteq Z(G)$ and thus
%the proof of Proposition~3.1 in \cite{varphi}
%remains valid).
%
%\begin{proposition}\cite[Proposition 3.1]{varphi}\label{LCproperty}
%Let $R$ be a commutative ring with $\Char(R)=4$. Let
%$G$ be a group with the $LC$--property and unique
%non-trivial  commutator $s$ and with an involution
%$\vp$ given by $\vp(h)=sh$ if $g\in G\setminus Z(G)$
%and $\vp (g)=g$ if $g\in Z(G)$. Then
%$(RG)^{-}_\varphi$ is commutative.
%\end{proposition}

From the next lemma it follows that if $R_{2}\neq \{
0\}$ and $\Char (R)\neq 4$ then any two elements of
$G\setminus G_{\vp}$ that satisfy condition $(1)$ of
Lemma~\ref{lema1} must commute.

\begin{lemma}\label{charpar}
Assume $R_{2}\neq \{ 0\}$ and $(RG)_{\vp}^{-}$ is
commutative. Let $g,h\in G$ and suppose $(g,h)\neq
1$.
\begin{enumerate}
\item If $g\in G_\vp$ and
 $h\not\in G_\vp$ then $gh=\vp(h)g$ and $hg=g\vp(h)$.
\item If $g,h\in G\setminus G_{\vp}$ then $gh\not\in
G_{\vp}$ (in particular, $g$ and $h$ do not satisfy
condition $(1)$ of Lemma~\ref{lema1}).
\item
If $g,h\in G\setminus G_\vp$ then $\Char(R)=4$,
$\GEN{g,h}$ is LC with a unique non-trivial
commutator and $\vp(g)=(g,h)g$ and $\vp(h)=(g,h)h$.
In particular, if $\Char(R)\neq 4$ then $K=\GEN{g\in
G\mid g\not\in G_\vp}$ is abelian.
\end{enumerate}
\end{lemma}
\begin{proof}
Let $0\neq r\in R_{2}$.

$(1)$   Since $(RG)_{\vp}^{-}$ is commutative, we
have that $0=[rg,
h-\vp(h)]=r(gh+g\vp(h)+hg+\vp(h)g)$. As $(g,h)\neq
1$ and $h\not\in G_{\vp}$, it follows that
$gh=\vp(h)g$ and $hg=g\vp(h)$.

$(2)$ Suppose  $g,h\in G\setminus G_{\vp}$. Assume
$gh \in G_\vp$. Then, by $(1)$,
$\vp(h)\vp(g)h=ghh=\vp(h)gh$ and thus $g\in G_\vp$,
a contradiction.

(3) This follows at once from  Lemma~\ref{lema1},
$(2)$ and Lemma~\ref{lema4} .
\end{proof}

We now give a  complete characterization of when
$(RG)^-_\vp$ is commutative provided $R_{2}\neq \{
0\}$ (and thus $\Char (R)\neq 3$).

\begin{theorem} \label{resultchar4}
Let $R$ be a commutative ring with  elements of
additive order $2$. Assume  $G$ is  a non-abelian
group and $\varphi$ is an involution on $G$. Then,
$(RG)^-_\vp$ is commutative if and only if one of
the following conditions holds:
\begin{enumerate}
\item[(a)] \label{KK}
$K=\GEN{g\in G|\; g\not\in G_\vp}$ is abelian (and
thus $G=K\cup Kx$, where $x\in G_\vp$, and
$\vp(k)=xkx^{-1}$ for all $k\in K$), and
  $R_{2}^{2}= \{ 0\}$.
\item[(b)]
$\Char(R)=4$, $|G'|=2$, $G/G'=(G/G')_\vp$, $g^2\in
G_\vp$ for all $g\in G$, and $G_{\vp}$ is
commutative in case  $R_{2}^{2}\neq \{ 0\}$.
\end{enumerate}
\end{theorem}

\begin{proof}
 Let $G$ be a non-abelian group and $\vp$ an
involution on $G$. Assume  $(RG)_{\vp}^{-}$ is
commutative. Notice that Lemma~\ref{K} implies that
if $K=\GEN{g\in G|\; g\not\in G_\vp}$ is abelian
(and thus $K\neq G$) then $G=K\cup Kx$ for some
$x\in G_{\vp}$. Furthermore, one gets that
$\vp(k)=xkx^{-1}$ for all $k\in K$. Indeed, since
$x\not\in K$ it follows that $xk\not\in K$ and hence
$xk=\vp(xk)=\vp(k)x$ and therefore
$\vp(k)=xkx^{-1}$.   Also, since $x$ is not central,
we get that $xk\neq kx$ for some $k\in K$. Now, for
any $r_{1},r_{2}\in R_{2}$ we have that $x,kx\in
G_{\vp}$ and thus, by assumption,
$r_{1}r_{2}(xkx-kx^{2})=[r_{1}x,r_{2}kx]=0$. Since
$xkx\neq kx^{2}$, it follows that $r_{1}r_{2}=0$.
Consequently, $R_{2}^{2}=\{ 0\}$. So, condition (a)
follows.

If $\Char (R)\neq 4$ then it follows from
Lemma~\ref{charpar}.(3) that $K$ is abelian. Hence,
by the above, condition (a) follows.

So, to prove the necessity of the mentioned
conditions, we are left to deal with the case that
$\Char (R)=4$ and  $K$ is not abelian. We need to
prove that condition (b) holds. Because of
Lemma~\ref{charpar} $(3)$, we also know that
$H=\GEN{x,y}$ is LC with a unique non-trivial
commutator and $\vp (h)=(x,y)h$ if $h\in H\setminus
Z(H)$ and $\vp (h)=h$ if $h\in Z(H)$.

Now we claim  that for all $g\in G\setminus G_\vp$
we have that $g^2\in G_\vp$ and $g^{-1}\vp(g)=(x,y)$
(in particular, $G/G' = (G/G')_{\vp}$). Indeed, let
$g\in G\setminus G_\vp$. If $(g,x)\neq 1$  then by
Lemma~\ref{charpar} $(3)$ $g^2\in G_\vp$ and
$g^{-1}\vp(g)=x^{-1}\vp(x)=(x,y)$. Similarly, if
$(g,y)\neq 1$ then $g^{-1}\vp (g) = y^{-1} \vp (y) =
(x,y)$. Assume now that $(g,x)=(g,y)=1$. If $gx\in
G_\vp$ then $g^2x^2 = (gx)^2 = \vp((gx)^2) =
\vp(g^2x^2) = \vp(g^2)x^2$ and hence $g^2\in G_\vp$.
Moreover, in this case, $gx=\vp(g)\vp(x)$ and hence
$g^{-1}\vp(g)=x\vp(x^{-1})=x^{-1}\vp(x)=(x,y)$ as
desired. If $gx\not\in G_{\vp}$ then, by
Lemma~\ref{charpar} $(3)$ and since $(gx,y)=(x,y)
\neq 1$, we get that $(gx)^{-1}\vp (gx) = (gx,y) =
(x,y) =x^{-1}\vp (x)$. Hence, $x^{-1}g^{-1}\vp (x)
\vp (g) =x^{-1}\vp (x)$ and thus $g^{-1} \vp (g) \vp
(x) = g^{-1} \vp (x) \vp (g) =\vp (x)$. So, $g=\vp
(g)$, a contradiction. This finishes the proof of
the claim.

Next we show that $G'=\GEN{(x,y)}=\{ 1, (x,y)\}$
(and thus $G'\subseteq G_{\vp}$). Indeed, let
$g,h\in G$ such that $(g,h)\neq 1$. If $g,h\not\in
G_\vp$ then by the previous claim and
Lemma~\ref{charpar} $(3)$ it follows that $(g,h) =
g^{-1}\vp(g) = x^{-1}\vp(x) = (x,y)$, as desired. If
$g\in G_\vp$ and $h\not\in G_\vp$ then, by
Lemma~\ref{charpar} $(1)$, $gh=\vp(h)g$ and hence by
the previous claim we get that $(x,y) = \vp(x)x^{-1}
= \vp(h)h^{-1} = ghg^{-1}h^{-1} = (g,h)$. Finally if
$g,h\in G_\vp$ then $hg\not\in G_\vp$ (because
otherwise $(g,h)=1$), and hence by the previous
claim $(x,y) = \vp(x) x^{-1} = \vp ( (hg)) (hg)^{-1}
= gh g^{-1}h^{-1} = (g,h)$, as desired.

To finish the prove of the necessity, we remark that
if $R_{2}^{2}\neq \{ 0\}$ then $G_{\vp}$ is
commutative. Indeed, let $r_{1},r_{2}\in R_{2}$ be
so that $r_{1}r_{2}\neq 0$ and let $g_{1},g_{2}\in
G_{\vp}$. Since $(RG)_{\vp}^{-}$ is commutative, we
have that $r_{1}r_{2}(g_{1}g_{2}-g_{2}g_{1}) =
[r_{1}g_{1},r_2 g_{2}]=0$. Hence $(g_{1},g_{2})=1$.

In order to prove the sufficiency we need  to show
that the elements in $$\mathcal{S}=\{g-\vp(g)\mid
g\in G, \;g\not\in G_\vp\; \}\cup\{rg\mid \;g\in
G_\vp,\;r\in R_2 \}$$ commute.

First assume $G$ satisfies condition $(a)$. So
$G=K\cup Kx$ with $x\in G_{\vp}$ and $K$ abelian. We
need to show that $[g-\vp (g), r_{1}h_{1}]=0$ and
$[r_{1}h_{1},r_{2}h_{2}]=0$ for   $g\in G\setminus
G_{\vp}$, $h_{1},h_{2} \in G_{\vp}$ and $r_{1},r_{2}
\in R_{2}$ with $(g,h_{1})\neq 1$ and
$(h_{1},h_{2})\neq 1$. The later equality is
obviously satisfied because of the assumptions. To
prove the former equality, we note that, by
Lemma~\ref{charpar} $(1)$, $h_{1}g=\vp(g)h_{1}$ and
$gh_{1}=h_{1}\vp(g)$. Hence,
\begin{eqnarray*}
[g-\vp(g),r_1h_{1}]
 &=&
r_1gh_{1}-r_1\vp(g)h_{1} - r_1h_{1} g + r_1h_{1}
\vp(g)\\ &=& r_1gh_{1} + r_1h_{1}g + r_1h_{1}g +
r_1gh_{1} \\&=& 2r_1gh_{1} + 2r_1h_{1}g\\& =& 0,
\end{eqnarray*} as desired.

Second, assume  $G$ satisfies  $(b)$ and that
$\Char(R)=4$. Notice that in this case if $g\not \in
G_\vp$ then $g^{-1}\vp(g)=g\vp(g^{-1})$ is central
and equal to the unique commutator of $G$. Let
$g,h\in G$ with $(g,h)\neq 1$ and let
$r_{1},r_{2}\in R_{2}$. If $g,h\in G_{\vp}$, then
the assumptions imply that $R_{2}^{2}=\{ 0\}$ and
thus $[r_{1}g,r_{2}h]=0$, as desired.  If $g\not\in
G_\vp$ and $h\in G_\vp$ then $$
\begin{array}{lcl}
[g-\vp(g),r_{1}h]&=&
r_{1}gh-r_{1}\vp(g)h-r_{1}hg+r_{1}h\vp(g)\\
 &=&r_{1}gh+r_{1}hg\vp(g^{-1})
\vp(g)+r_{1}hg-r_{1}\vp(g)\vp(g^{-1})gh\\
&=&r_{1}gh+r_{1}hg+r_{1}hg+r_{1}gh\\
 &=&2r_{1}hg+2r_{1}gh=0.
\end{array}
$$ Finally, if $g,h\in G\setminus G_{\vp}$ then $$
\begin{array}{lcl}
[g-\vp(g),h-\vp(h)]&=&gh-g\vp(h)-\vp(g)h+\vp(g)\vp(h)-hg+h\vp(g)+\vp(h)g-\vp(h)\vp(g)\\
%&=&gh-\vp(h)\vp(h^{-1})hg-hg\vp(g^{-1})\vp(g)+gh-hg+\vp(g)\vp(g^{-1})gh+gh\vp(h^{-1})\vp(h)-hg\\
&=&gh-hg-hg+gh-hg+gh+gh-hg=4gh-4hg=0
\end{array}
$$ which finishes the proof of the theorem.
\end{proof}

For the classical involution $*$ on $G$ we get the
following consequence.

\begin{corollary}
Let $R$ be a commutative ring with  elements of
additive order $2$. Let $G$ be a non-abelian group.
Denote by $^*$ the classical involution. Then
$(RG)^-_*$ is commutative if and only if one of the
following conditions holds:
\begin{enumerate}
\item $G=K\rtimes \GEN{x}$ where $K=\langle g \mid g ^2\neq1 \rangle$, $K$ is abelian, $x^2=1$, $xkx=k^{-1}$
for all $k\in K$ and  $R_{2}^{2}= \{ 0\}$.
\item $\Char(R)=4$,  $G$ has exponent $4$, $G'$ is a cyclic
group of order $2$, $G/G'$ is an elementary abelian
2-subgroup and elements of order $2$ commute if
$R_{2}^{2}\neq \{ 0\}$.
\end{enumerate}
\end{corollary}

\section{Rings of characteristic three}

In this section we determine when $(RG)_{\vp}^{-}$
is commutative if $\Char (R)=3$  (and thus $R_{2}=
\{ 0 \}$). Again we begin by recalling two technical
lemmas from \cite{varphi}.

\begin{lemma}{\em \cite[Lemma 1.3]{varphi}}\label{lemg2}
Let $R$ be a commutative ring with $char(R)\neq 2$
and let $g\in G\setminus G_\vp$. If $(RG)^-_\vp$ is
commutative then one of the following conditions
holds:
\begin{enumerate}
\item $g\vp(g)=\vp(g)g$.
\item $g^2\in G_\vp$.
\end{enumerate}
\end{lemma}

\begin{lemma}{\em\cite[Lemma 1.1]{varphi}}\label{lema2}
Let $R$ be a commutative ring with $\Char (R) = 3$.
Let $g,h\in G\setminus G_{\vp}$ be two non-commuting
elements. If $(RG)^{-}_{\vp}$ is commutative then
one of the following conditions holds
\begin{enumerate}
\item $gh\in G_{\vp}$, $hg\in G_{\vp}$, $h\vp (g)=g\vp (h)$ and
$\vp (h)g=\vp (g) h$.
\item $gh\in G_{\vp}$, $hg\in
G_{\vp}$, $h\vp (g)=\vp (g)h$.
\item $gh\in
G_{\vp}$, $hg=\vp (g)h=g\vp (h)$.
\item $gh=h\vp(g)=\vp(g)\vp(h)$, $\vp(h)g=\vp(g)h$.
\item $gh=\vp(h)g=\vp(g)\vp(h)$,
$h\vp(g)=g\vp(h)$.
\item $hg\in G_{\vp}$, $gh=\vp (h)g=h\vp (g)$.
%\item $gh=h\vp(g)=\vp(h)g=\vp(g)\vp(h)$ and $\Char(R)=4$.
\end{enumerate}
\end{lemma}

The following lemma was proved in \cite{varphi} in
the case when $\Char (R)$ is distinct from both $2$
and $3$.

\begin{lemma} \label{onethentwo}
Let $R$ be a commutative ring. Let $g,h\in
G\setminus G_\varphi$ be non-commuting elements that
satisfy $(2)$ of Lemma~\ref{lema2}. If $char(R)=3$
then $g,h$ also satisfy $(1)$ of Lemma~\ref{lema2}.
\end{lemma}
\begin{proof}
Consider the element $\vp(g)hg\in G$. Since $h\not
\in G_\vp$ we have that $\vp(g)hg\not\in G_\vp$.
Also $\vp(g)hg$ and $h$ do not commute because, by
assumption, $h\vp(g)=\vp(g)h$ and $gh\neq hg$.
Assume that $g$ and $h$ satisfy (2) of
Lemma~\ref{lema2}. We claim that then $\vp
(h)g=\vp(g)h$.

We deal with two mutually exclusive cases. First,
assume that $\vp(g)hgh\in G_\vp$, i.e.,
$\vp(g)hgh=\vp(\vp(g)hgh)=\vp(h)\vp(g)\vp(h)g=\vp(h)hg^2$
since $hg\in G_\vp$. If $(g^2,h)=1$ we obtain that
$\vp(g)h=\vp(h)g$, as desired. So, to deal with this
case, we may assume that $(g^2,h)\neq 1$. If $g^2\in
G_\vp$ then, using (2),  we observe that $\vp (h)
hg^{2}=
\vp(g)hgh=\vp(g^2)\vp(h)h=g^2\vp(h)h=\vp(h)g^2h$.
Hence, we get that $(g^2,h)=1$, a contradiction. In
the rest of the proof we will several times use
(without referring to this) that $gh,\, hg\in G_\vp$
and $(g,\vp(h))=1=(h,\vp(g))$. Moreover, since
$g^2\not\in G_\vp$, we also have that
$g\vp(g)=\vp(g)g$, by Lemma~\ref{lemg2}.

So, if $\vp(g)hgh\in G_\vp$ then we may assume that
$g^2\not\in G_\vp$ and $(g^2,h)\neq 1$.  Hence,
$g^2$ and $h$ satisfy one of the six conditions of
Lemma~\ref{lema2}. We now show that this situation
can not occur.  Assume first that $g^2 h\in G_\vp$.
Then,
$g^2h=\vp(g^2h)=\vp(h)\vp(g^2)=gh\vp(g)=g\vp(g)h$
and thus   $g\in G_\vp$, a contradiction. Therefore
$g^2$ and $h$ do not satisfy conditions $(1)-(3)$ of
Lemma~\ref{lema2}. If $g^2$ and $h$ satisfy $(4)$ of
Lemma~\ref{lema2} then $g^2h=h\vp(g^2)=\vp(g^2)h$
and hence $g^2\in G_\vp$, a contradiction. Finally,
if $g^2$ and $h$ satisfy either $(5)$ or $(6)$ of
Lemma \ref{lema2} then
$g^2h=\varphi(h)g^2=g^2\varphi(h)$ and thus $h\in
G_\varphi$, a contradiction. This finishes the proof
of the first case.

Second, assume that $\vp(g)hgh\not\in G_\vp$. Then
$\vp(g)hg$ and $h$ satisfy one of the conditions
$(4)-(6)$ of Lemma~\ref{lema2}. We show that all
these lead to a contradiction and hence that this
case also can not occur. If $\vp(g)hg$ and $h$
satisfy $(4)$ of Lemma~\ref{lema2}, then $\vp(g)hgh
= h\vp( \vp (g) hg) = h\vp (g) \vp (h) g = h\vp (hg)
g= h^2g^2$ and thus $\vp(g)gh=hg^2=\vp(g)\vp(h)g$;
so $gh=\vp(h)g$ and hence $g\in G_\vp$, a
contradiction.

Suppose that $\varphi(g)hg$ and $h$ satisfy $(5)$ or
$(6)$ of Lemma~\ref{lema2}. Then
\begin{equation}\label{str}
\varphi(g)hgh=\varphi(h)\varphi(g)hg
\end{equation}
First assume that $g^2\in G_\varphi$ then we have
that
$\varphi(h)hg^2=g^2\varphi(h)h=\varphi(g^2)\varphi(h)h=
\varphi(g)hgh$. On the other hand
$\varphi(h)\varphi(g)hg=\varphi(h)h\varphi(g)g$.
Thus, by (\ref{str}) we get that $g\in G_\varphi$, a
contradiction. Therefore $g^2\not\in G_\varphi$ and
hence, by Lema~\ref{lemg2} we get that
$(g,\varphi(g))=1$. If also $(h,\varphi(h))=1$ then
$\varphi(g)hgh=hg\varphi(g)h$ and on the other hand,
$\varphi(h)\varphi(g)hg=hg\varphi(h)\varphi(g)$.
Then, by (\ref{str}), we get that
$\varphi(g)h=\varphi(h)\varphi(g)=gh$ and thus $g\in
G_\varphi$, a contradiction. So, again by
Lemma~\ref{lemg2} we have that $h^2\in G_\varphi$.
Therefore $\varphi(h)\varphi(g)hg=gh^2g=h^2g^2$ and
on the other hand $\varphi(g)hgh=h\varphi(g)gh$.
Thus, by (\ref{str}), we have that
$\varphi(g)gh=hg^2=\varphi(g)\varphi(h)g$. Therefore
$\varphi(h)\varphi(g)=gh=\varphi(h)g$ and hence
$g\in G_\varphi$, again a contradiction. So
$\varphi(g)hg$ and $h$ do not satisfy neither $(5)$
nor $(6)$ of Lemma~\ref{lema2}.

So, we have proved that if (2) of  Lemma \ref{lema2}
holds for non-commuting elements $g,h\in G\setminus
G_{\varphi}$ then $\varphi (h)g=\varphi (g)h$. Since
$(h,\varphi (g)) =1=( g,\varphi (h))$ it also
follows that $h\varphi(g)=g\varphi(h)$.
Consequently, we have shown that (1) of Lemma
\ref{lema2} holds for $g$ and $h$.
\end{proof}

\begin{lemma}\label{g3}
Let $R$ be a commutative ring. Let $g,h\in
G\setminus G_\varphi$ be non-commuting elements.
\begin{enumerate}
\item If $g$ and $h$ satisfy $(1)$ (or $(2)$) of
Lemma~\ref{lema2} then  $g^3,h^3\not\in G_\vp$
\item If $g$ and $h$ satisfy one of the conditions $(3)-(6)$ of Lemma~\ref{lema2}
then $g^3,h^3\in G_\vp$ and  $g^3, h^3\in Z(\langle
g,h,\vp(g),\vp(h)\rangle)$.
\end{enumerate}
\end{lemma}

\begin{proof}
$1.$ Let $g,h\in G\setminus G_\vp$ be non-commuting
elements. Assume that $g$ and $h$ satisfy (1) of
Lemma~\ref{lema2}. We prove by contradiction that
$g^{3}\not\in G_{\vp}$. So, suppose that $g^3\in
G_\vp$. Since $g \not\in G_\vp$, it follows that
$g^2 \not\in G_\vp$. Also by $(1)$ of
Lemma~\ref{lema2} we have that
\begin{equation}\label{eq}
g^3h=g^2\vp(h)\vp(g)=gh\vp(g^2)=\vp(h)\vp(g^3)=\vp(h)g^3
\end{equation}
%{\bf we should give another label to the equation as
%it is confusing with numbering in enumerates}
%Since $g$ and $h$ play a symmetric role in $(1)$ of
%Lemma~\ref{lema2} we will only show the result for
%$g$.
Notice that by (\ref{eq}) it follows that
$(g^2,h)\neq 1$, because otherwise $gh=\vp(h)g$, a
contradiction. Therefore $g^2$ and $h$ satisfy one
of the  conditions $(1)-(6)$ of Lemma~\ref{lema2}.
Assume first that $g^2h\in G_\vp$. Then by
(\ref{eq}) we have that
$g\vp(h)\vp(g^2)=gg^2h=\vp(h)\vp(g^3)$. Hence
$g\vp(h)=\vp(h)\vp(g)=gh$ and thus $h\in G_\vp$, a
contradiction. Therefore $g^2$ and $h$ do not
satisfy  conditions $(1)-(3)$ of Lemma~\ref{lema2}.
Second, assume that $g^2h=\vp(h)g^2$. Then, by
(\ref{eq}), it follows that
$\vp(h)g^3=gg^2h=g\vp(h)g^2$ and thus
$(g,\vp(h))=1$. Therefore, again by (\ref{eq}), we
get that $h\in G_\vp$, a contradiction. So, $g^2$
and $h$ do not satisfy conditions $(5)-(6)$. Hence,
$g^2$ and $h$ satisfy $(4)$. Then, since
$(g,\vp(g))=1$ by Lemma~\ref{lemg2}, we have that
$\vp(g^3)h=gg^2h=gh\vp(g^2)=g\vp(g^2)\vp(h)=\vp(g^2)g\vp(h)$.
Hence $\vp(g)h=g\vp(h)=h\vp(g)$. Consequently, by
(\ref{eq}), we get that $h\in G_\vp$, a
contradiction. This finishes the proof of the fact
that $g^{3}\not\in G_{\vp}$. Because of the symmetry
in $g$ and $h$ in condition (1) of
Lemma~\ref{lema2}, we thus also obtain that
$h^3\not\in G_{\vp}$.

$2.$ Notice that if in $(3)$ of Lemma~\ref{lema2} we
interchange the roles of $g$ and $h$ then we obtain
$(6)$, if we change  $h$ by $\vp(h)$ we have $(5)$
and finally if we change $g$ by $\vp(g)$  we have
$(4)$. Therefore it is enough to show the result for
$(3)$.

So, assume that $g,h\in G\setminus G_{\vp}$ are
non-commuting elements that satisfy $(3)$ of
Lemma~\ref{lema2}. Then
$g^3h=g^2\vp(h)\vp(g)=g\vp(g)h\vp(g)=g\vp(g^2)\vp(h)$
and therefore, since $h\not\in G_\vp$, it follows
that $g^2\not \in G_\vp$. Thus, by
Lemma~\ref{lemg2}, we get that $(g,\vp(g))=1$.
Consequently, $g^3h=\vp(g^2)g\vp(h)=\vp(g^3)h$ and
therefore $g^3\in G_\vp$. Analogously we obtain that
$h^3\in G_\varphi$. Moreover, $g^3h=g^2\vp(h)\vp(g)=
ghg\vp(g)=\vp(h)\vp(g)g\vp(g) = \vp(h)g\vp(g)\vp(g)
= h\vp(g^3)=hg^3$. So, $g^{3}h=hg^{3}$ and thus also
$\vp (h) g^{3}=g^{3}\vp (h)$, as desired. Similarly
we get that $h^3\in Z(\langle
g,h,\vp(g),\vp(h)\rangle)$.
\end{proof}

\begin{remark}\label{remark}
Notice that Lemma~\ref{g3} implies that  if
$g,h,x,y\in G\setminus G_{\vp}$ are such that $g$
and $h$ are non-commuting elements satisfying
condition $(1)$ of Lemma~\ref{lema2} and $x$ and $y$
are non-commuting elements satisfying one of the
conditions $(3)-(6)$ of Lemma~\ref{lema2} then $x$
and $y$ commute with both $g$ and $h$.
\end{remark}

\begin{lemma}\label{key}
Let $R$ be a commutative ring with $\Char(R)=3$ and
assume $RG^{-}_{\vp}$ is commutative. If there exist
non-commuting $g,h\in G\setminus G_{\vp}$ so that
$g$ and $h$ satisfy $(1)$ of Lemma~\ref{lema2} then
all  $x,y\in G\setminus G_{\vp}$ satisfy $(1)$ of
Lemma~\ref{lema2}.
\end{lemma}

\begin{proof}
Let $g,h\in G\setminus G_{\varphi}$ be non-commuting
elements  so that $g$ and $h$ satisfy $(1)$ of
Lemma~\ref{lema2}, that is, $gh,\, hg,\,
g\varphi(h)$ and $\varphi(g)h$ are elements of
$G_{\varphi}$. Also, by Lemma~\ref{g3}, we have that
$g^3,\, h^3 \not\in G_\varphi$.

Let $x\in G\setminus G_\varphi$. Then, by
Lemma~\ref{lema2} and Lemma~\ref{g3}, it follows
that $(g,x)=1$ or $gx\in G_\varphi$. We claim that
$gx\in G_\varphi$. In order to prove this  claim
suppose that $gx\not\in G_\varphi$ and thus
$(g,x)=1$. Again, by Lemma~\ref{lema2} and
Lemma~\ref{g3}, it follows that $(h,x)=1$ or $xh\in
G_\varphi$; and $(gx,h)=1$ or $gxh\in G_\varphi$.
Assume first that $(gx,h)=1$, that is, $gxh=hgx$.
Since $gh\neq gh$ we get that $hx\neq xh$ and thus
$xh\in G_\varphi$. Therefore,
$hgx=gxh=g\varphi(h)\varphi(x)=h\varphi(g)\varphi(x)=h\varphi(g
x)$ and thus $gx\in G_\varphi$, a contradiction.
Second assume that $gxh\in G_\varphi$. Then
$gxh=\varphi(h)\varphi(x)\varphi(g)=\varphi(h)\varphi(g)\varphi(x)=gh\varphi(x)$
and hence $xh=h\varphi(x)$. Therefore, and since
$x\not\in G_\varphi$, we get that $xh\neq hx$ and
thus $xh\in G_\varphi$. Then
$\varphi(h)\varphi(x)=xh=h\varphi(x)$ and hence
$h\in G_\varphi$, again a contradiction. This
finishes the proof of the claim.

Now, let $x,\,y\in G\setminus G_\varphi$. We need to
prove that $x$ and $y$ satisfy $(1)$ of Lemma
\ref{lema2}. First we deal with the case that
$(x,y)\neq 1$. Because of Lemma~\ref{onethentwo}, we
only have to show that it is impossible that $x$ and
$y$ satisfy one of the conditions $(3)-(6)$ of
Lemma~\ref{lema2}.  So suppose the contrary. Then,
by Remark~\ref{remark}, $(g,x)=1=(g,y)$. Also, by
Lemma~\ref{g3},  $x^3,\, y^3 \in G_\varphi$. By the
previous claim we have that $gx\in G_\varphi$.
Consequently, $g^3x^3 = (gx)^3 = \varphi(gx)^3 =
\varphi(g^3)\varphi(x^3) = \varphi(g^3)x^3$, and
thus $g^3\in G_\varphi$, a contradiction. So, if
$xy\neq yx$ then $x$ and $y$ satisfy $(1)$ of
Lemma~\ref{lema2}.

Finally, assume $x,y\in G\setminus G_\varphi$ and
$(x,y)=1$. Then $xy\in G_\varphi$. Indeed, suppose
the contrary, that is assume $xy\not\in G_\varphi$.
Hence, by the above claim, $gxy\in G_\varphi$. Thus
$gyx=gxy=\varphi(y)\varphi(x)\varphi(g)=\varphi(y)gx$,
because $gx\in G_\varphi$. Therefore
$gy=\varphi(y)g$. Since $gy\in G_\varphi$, it
follows that $\varphi(y)g=gy=\varphi(y)\varphi(g)$
and thus $g\in G_\varphi$, a contradiction. Hence,
indeed $yx=xy\in G_\varphi$. Replacing $y$ by
$\varphi (y)$ we thus also get that $x\varphi (y)
\in G_{\varphi}$ if $(x,\varphi (y))=1$.  If, on the
other hand, $(x,\varphi (y))\neq 1$ then the
previous implies that again $x\varphi (y)\in
G_{\varphi}$. Similarly, $\varphi(x)y\in G_\varphi$.
Consequently, we have shown that $x$ and $y$ satisfy
$(1)$ of Lemma~\ref{lema2}.
\end{proof}

\begin{lemma}\label{comm}
Let $R$ be a commutative ring with $\Char(R)=3$. Let
$g,h\in G\setminus G_\vp$ be  non-commuting elements
satisfying any of the conditions $(3)-(6)$ of
Lemma~\ref{lema2}. Then $\langle
g^{-1}\vp(g)\rangle=\langle
h^{-1}\vp(h)\rangle=\langle (g,h)\rangle$ and
$(g,h)^3=1$.
\end{lemma}

\begin{proof}
Let $g,h\in G\setminus G_\vp$ be as in the statement
of the Lemma. Because of Lemma~\ref{g3}, $g^3,h^3\in
G_\vp$. Therefore $g^2, h^2\not\in G_\vp$, because
$g,h\not\in G_\vp$. Hence, by Lemma~\ref{lemg2}, it
follows that $(g,\vp(g))=1=(h,\vp(h))$.

First, assume that $g$ and $h$ satisfy $(3)$ of
Lemma~\ref{lema2}. Hence, $\vp(g)=hgh^{-1}$ and
$\vp(h)=g^{-1}hg$. Therefore
$g^{-1}\vp(g)=g^{-1}hgh^{-1}=\varphi(h)h^{-1}=h^{-1}\vp(h)$.
Thus,
$h^{-1}\vp(h)=g^{-1}\vp(g)=\vp(g)g^{-1}=hgh^{-1}g^{-1}=(g,h)^{-1}$,
as desired.

Second, assume that $g$ and $h$ satisfy $(4)$ of
Lemma~\ref{lema2}. Then $\vp(g)=h^{-1}gh$ and
$\vp(h)=\vp(g^{-1})gh=h^{-1}g^{-1}hgh$. Therefore
$(g,\vp(g))=(g,h^{-1}gh)=1=(h,\vp(h))=(h,g^{-1}hg)$.
Thus,
$g^{-1}\vp(g)=g^{-1}h^{-1}gh=h^{-1}ghg^{-1}=ghg^{-1}h^{-1}=(g,h)$
and
$h^{-1}\vp(h)=h^{-1}g^{-1}hg=(g^{-1}\vp(g))^{-1}$,
again as desired.

Third, assume that $g$ and $h$ satisfy $(5)$ of
Lemma~\ref{lema2}. Then $\vp(h)=ghg^{-1}$ and
$\vp(g)=gh\vp(h)^{-1}=ghgh^{-1}g^{-1}$. Therefore
$(h,\vp(h))=(h,ghg^{-1})=1=(g,\vp(g))=(g,hgh^{-1})$.
Thus, $g^{-1}\vp(g)=hgh^{-1}g^{-1}=(g,h)^{-1}$ and
$h^{-1}\vp(h)=h^{-1}ghg^{-1}=ghg^{-1}h^{-1}=(g^{-1}\vp(g))^{-1}$,
again as desired.

Fourth,  assume that $g$ and $h$ satisfy $(6)$ of
Lemma~\ref{lema2}. Then $\vp(h)=ghg^{-1}$ and
$\vp(g)=h^{-1}gh$. Therefore
$(h,\vp(h))=(h,ghg^{-1})=1=(g,\vp(g))=(g,h^{-1}gh)$.
Thus,
$g^{-1}\vp(g)=g^{-1}h^{-1}gh=h^{-1}ghg^{-1}=ghg^{-1}h^{-1}=(g,h)$
and
$h^{-1}\vp(h)=h^{-1}ghg^{-1}=ghg^{-1}h^{-1}=g^{-1}\vp(g)$,
as desired.

To finish the proof of the lemma notice that since
$g^3\in G_\vp$ and $(g,\vp(g))=1$ it follows that
$(g^{-1}\vp(g))^3=1$ and therefore $(g,h)^3=1$.
\end{proof}

\begin{theorem} \label{resultchar3}
Let $R$ be a commutative ring with $\Char(R)=3$.
Suppose $G$ is a non-abelian group and $\vp$ is an
involution on $G$. Then $(RG)^{-}_\varphi$ is
commutative if and only if one of the following
conditions holds:
\begin{enumerate}
\item[(a)] $K=\GEN{g\in G|\; g\not\in G_\vp}$ is abelian,
$G=K\cup Kx$ where $x\in G_\vp$ and
$\vp(k)=xkx^{-1}$ for all $k\in K$.
\item[(b)] $G$ contains an abelian subgroup of index $2$ that is contained in $G_{\vp}$.
\item[(c)] $|G'|=3$, $(G/G')=(G/G')_\vp$ and $g^3\in G_\vp$ for all $g\in G$.
\end{enumerate}
\end{theorem}

\begin{proof}
Assume that there exist  non-commuting elements
$g,h\in G\setminus G_{\vp} $ so that $g$ and $h$
satisfy $(1)$ of Lemma~\ref{lema2}. Then, by
Lemma~\ref{key}, all  $x,y\in G\setminus G_{\vp}$
satisfy $(1)$ of Lemma~\ref{lema2}. Because of all
the stated Lemmas, one now obtains, exactly as in
the proof of Theorem~2.1 in \cite{varphi} that
condition (a) or (b) holds.

So now suppose that there do not exist non-commuting
elements $g,h\in G\setminus G_{\vp}$ satisfying
condition $(1)$ (and thus not $(2)$, by
Lemma~\ref{onethentwo} ) of Lemma~\ref{lema2}. Then,
all non-commuting elements $x,y\in G\setminus G_\vp$
satisfying one of the conditions $(3)-(6)$ of
Lemma~\ref{lema2}. In particular, by Lemma~\ref{g3},
$x^{3},y^{3}\in G_{\vp}$. Since $x\not\in G_{\vp}$,
we thus have that $x^{2}\not\in G_{\vp}$ and thus,
by Lemma~\ref{lemg2}, $(\vp(x),x)=1$.

We claim that $g^3\in G_\vp$ for all $g\in G$. So,
let $g\in G$. In case $(g,x)\neq 1$  then it follows
at once from Lemma~\ref{g3} that $g^{3}\in G_{\vp}$.
If, on the other hand, $(g,x)=1$, then we consider
two mutually exclusive cases. First, assume
$gx\not\in G_{\vp}$. Then, again by Lemma~\ref{g3},
$g^{3}x^{3}=(gx)^{3}\in G_{\vp}$ and thus
$g^{3}x^{3}=\vp (g^{3})\vp (x^{3}) =\vp (g^{3})
x^{3}$ and thus $g^{3}\in G_{\vp}$. Second, assume
$gx\in G_{\vp}$. Thus $xg=gx =\vp (x) \vp (g) =\vp
(g) \vp (x)$. Hence, by Lemma~\ref{comm}, $g^{-1}
\vp (g) = x\vp (x)^{-1}=\vp (x)^{-1}x=\vp (g)g^{-1}$
is an element of order $3$. Thus $1=(g^{-1}\vp
(g))^{3} = g^{-3}\vp (g^{3})$. So $g^{3}\in
G_{\vp}$, as claimed.

Next, we claim that if $g\not\in G_\vp$ then
$g^{-1}\vp(g)\in \langle (x,y)\rangle$ (in
particular $G/G'=(G/G')_\vp$). Indeed, if $(g,x)\neq
1$, $(g,\vp (x))\neq 1$, $(g,\vp (y))\neq 1$ or
$(g,y)\neq 1$ the result follows from
Lemma~\ref{comm}. So assume that $(g,x)=(g,\vp (x))
= (g,y)= (g,\vp (y))=1$. If $gx\in G_\vp$ then
$gx=\vp(gx)=\vp(x)\vp(g)=\vp(g)\vp(x)$ and hence
$g^{-1}\vp(g)=x\vp(x)^{-1}\in \langle (x,y)\rangle$
by Lemma~\ref{comm} as desired. Finally if
$gx\not\in G_\vp$ then $(gx,y)=(x,y)\neq 1$. Then,
by Lemma~\ref{comm}, we have that
$(gx)^{-1}\vp(gx)=g^{-1}\vp(g)x^{-1}\vp(x)\in
\langle (x,y)\rangle$ and therefore, since
$x^{-1}\vp(x)\in \langle (x,y)\rangle$ we get that
$g^{-1}\vp(g)\in \langle (x,y)\rangle$, which
finishes the proof of the claim.

To finish the proof of the necessity, we have to
prove that $G'=\langle (x,y)\rangle$ and thus, by
Lemma~\ref{comm}, $|G'|=3$. In order to prove this,
let $g,h\in G$ with $(g,h)\neq 1$. If $g,h\not\in
G_\vp$ then by Lemma~\ref{comm} and the previous
claim we have that $(g,h)\in \langle
g^{-1}\vp(g)\rangle=\langle (x,y)\rangle$. Next
assume that $g\in G_\vp$  and $h\not\in G_\vp$. If
$gh\not\in G_\vp$ then, by Lemma~\ref{comm} and the
previous claim, we have that $(g,h)=(gh,h)\in
\langle h^{-1}\vp(h)\rangle=\langle (x,y)\rangle$,
as desired. If $gh\in G_\vp$ we have that
$gh=\vp(h)g$ and thus, by the previous claim,
$ghg^{-1}h^{-1}=\vp(h)h^{-1}\in \langle
(x,y)\rangle$, as desired. Finally, assume that
$g,h\in G_\vp$. Then, since $(g,h)\neq 1$, it
follows that $h^{-1}g^{-1}\not\in G_\vp$. Therefore,
by the previous claim, $(g,h) = ghg^{-1}h^{-1} =
(h^{-1}g^{-1})^{-1}\vp(h^{-1}g^{-1})\in \langle
(x,y)\rangle$ and the proof of the necessity
concludes.

In order to prove the sufficiency, we need  to show
that the elements in $\mathcal{S}=\{g-\vp(g)\mid
g\in G, \;g\not\in G_\vp \}$ commute. If $G$
satisfies conditions $(a)$ or $(b)$, then the proof
is the same as the sufficiency proof of Theorem~2.1
in  \cite{varphi}. So, assume that $G$ satisfies
condition $(c)$. Then, since $g^3\in G_\vp$, we have
that $g^2\not\in G_\vp$ for all $g\in G\setminus
G_\vp$ and hence, by Lemma~\ref{lemg2},
$(g,\vp(g))=1$. Moreover, it  follows from $(c)$
that $\vp(g)=t^ig$, where $G'=\langle t \rangle$,
$i\in \{\pm 1\}$ for $g\not\in G_\vp$. Then,
clearly, $(g,t)=1$.

Let now $g,h\in G\setminus G_{\vp}$. Then
$\vp(g)=t^ig$ and $\vp(h)=t^jh$ with $i,j\in\{\pm
1\}$ and $$
\begin{array}{lcl}
[g-\vp(g),h-\vp(h)]&=&gh-g\vp(h)-\vp(g)h+\vp(g)\vp(h)-hg+h\vp(g)+\vp(h)g-\vp(h)\vp(g)\\
&=&gh-t^jgh-t^igh+t^it^jgh-hg+t^ihg+t^jhg-t^it^jhg.\\
\end{array}
$$ If $(g,h)=1$ then clearly
$[g-\vp(g),h-\vp(h)]=0$. So, assume that
$1\neq(g,h)=t$. If $i=j$ then,
 since $\Char(R)=3$,
$$
\begin{array}{lcl}
[g-\vp(g),h-\vp(h)]&=&gh-t^igh-t^igh+t^{-i}gh-hg+t^ihg+t^ihg-t^{-i}hg\\
&=& (1-2t^{i}+t^{-i})gh-(1-2t^{i}+t^{-i})hg\\ &=&
(1+t^{i}+t^{-i}) (t-1)hg\\ &=& (1+t+t^{2})(t-1)gh\\
&=&0
\end{array} $$
On the other hand if $i\neq j$ then, again since
$\Char(R)=3$, $$
\begin{array}{lcl}
[g-\vp(g),h-\vp(h)]&=&gh-t^jgh-t^igh+t^it^jgh-hg+t^ihg+t^jhg-t^it^jhg\\
&=&2gh-t^jgh-t^igh-2hg+t^ihg+t^jhg\\ &=&
(2-t^{j}-t^{-i})gh-(2-t^{i}-t^{j})hg\\ &=&
(2-t^{j}-t^{i})(t-1)hg\\ &=& 2(1+t+t^{2})(t-1)hg\\
&=& 0
\end{array}
$$ Similarly, if $(g,h)=t^{-1}$ one  gets that
$[g-\vp(g),h-\vp(h)]=0$, which finishes the proof of
the theorem.
\end{proof}

\section*{Acknowledgements}
Research supported by Onderzoeksraad of Vrije
Universiteit Brussel, Fonds voor Wetenschappelijk
Onderzoek (Vlaanderen), Flemish-Polish bilateral
agreement BIL 01/31, FAPEMIG and CNPq. Proc.
300243/79-0(RN) of Brazil,  D.G.I. of Spain and
Fundaci\'on S\'eneca of Regi\'on de Murcia.

\noindent
\begin{tabular}{ll}
O. Broche Cristo & \hspace{1cm} Eric Jespers
\\ Dep. de Ci\^{e}ncias Exatas & \hspace{1cm} Dept. Mathematics  \\
Universidade Federal de Lavras & \hspace{1cm} Vrije Universiteit
Brussel
\\ Caixa Postal 3037& \hspace{1cm} Pleinlaan 2  \\
37200-000 Lavras, Brazil  & \hspace{1cm} 1050 Brussel, Belgium \\
osnel@ufla.br & \hspace{1cm} efjesper@vub.ac.be
\end{tabular}

 \vspace{12pt} \noindent
\begin{tabular}{ll}
C. Polcino Milies & Manuel Ruiz
\\ Instituto de Matem\'atica y Estat\'{\i}stica  & Dep.
M\'{e}todos Cuantitativos e Inform\'{a}ticos \\
 Universidaded de S\~{a}o Paulo &Universidad Polit\'{e}cnica de Cartagena
\\  Caixa Postal 66281&  Paseo Alfonso
XIII, 50 \\ 05311-970 S\~{a}o Paulo, Brazil  &
30.203 Cartagena, Spain
\\ polcino@ime.usp.br & Manuel.Ruiz@upct.es
\end{tabular}


\begin{thebibliography}{99}
\itemsep=-2pt

\bibitem{A}  S. A. Amitsur,  Identities in rings with involutions, {\it Israel J.
Math.} {\bf 7} (1969) 63--68.

\bibitem{osn1} O. Broche Cristo, Commutativity of symmetric elements in group rings, {\it J. Group Theory}, {\bf 9}, 5 (2006),  673--683

\bibitem{jes-osn-ruiz} O. Broche, E. Jespers, M.
Ruiz, Antisymmetric elements in group rings with an
orientation morphism, {\it Forum Mathematicum}, to appear.

\bibitem{osn2} O. Broche Cristo and C. Polcino Milies,
Commutativity of skew symmetric elements in group
rings, {\it Proc. Edinburgh Math. Soc.} (2) 50
(2007), no. 1, 37-47.


\bibitem{GM} A. Giambruno,  C.Polcino Milies, Unitary units and skew elements in group algebras.
{\it Manuscripta Math. } {\bf 111} (2003), no. 2,
195--209.

\bibitem{gia-seh} A. Giambruno, S.K. Sehgal, Lie nilpotence of
group rings, {\it Comm. Algebra} {\bf 21} (1993),
4253--4261.

\bibitem{gs} A. Giambruno and  S.K. Sehgal, Groups algebras whose Lie algebra
of skew-symmetric elements is nilpotent, Groups,
rings and algebras, 113--120, {\it Contemp. Math.}, 420,
Amer. Math. Soc., Providence, RI, 2006.


\bibitem{GoPas}  J.Z. Gon\c{c}alves and D.S.  Passman,  Unitary units in group algebras,
{\it Israel J. Math.}  {\bf 125} (2001), 131--155.

\bibitem{eric} E.G. Goodaire, E. Jespers, C. Polcino Milies,
{\it Alternative Loop Rings},  North Holland Math.
Studies 184, Elsevier, Amsterdam, 1996.

\bibitem{gupta} N. D. Gupta and F. Levin, On the Lie ideals of a ring, {\it Journal of Algebra},
 {\bf 81} (1), (1993), 225--231.

\bibitem{varphi} E. Jespers and M. Ruiz, Antisymmetric elements in
group rings, {\it J. Algebra and its Appl.}, {\bf
4}, 4 (2005), 341--353.

\bibitem{JR} E. Jespers, M. Ruiz, On symmetric elements and
symmetric units in group rings, {\it Comm. Algebra},
{\bf  34}, (2) (2006), 727-736.


\bibitem{novikov} S. P. Novikov,
Algebraic construction and properties of hermitian
analogues of
  $K$-theory over rings with involution from the viewpoint of Hamiltonian
  formalism,  Applications to differential topology and the theory of
  characteristitic classes II,
 Izv. Akad. Nauk SSSR Ser. Mat.  {\bf 34}  (1970), 475-500;
 English transl. in Math. USSR Izv., {\bf 4} (1970).

\bibitem{ZS}Zalesski\u\i, A. E.; Smirnov, M. B. Lie algebra associated with linear group. {\it Comm. Algebra} {\bf 9}
 (1981), 20, 2075--2100.


\end{thebibliography}
\end{document}